\documentclass[11pt]{article}
\usepackage[margin=1in]{geometry} 
\usepackage{amssymb,amsfonts,amsmath,bbm,mathrsfs,stmaryrd,mathtools}
\usepackage{xcolor}
\usepackage{url}

\usepackage{extarrows}

\usepackage{enumerate}

\usepackage[colorlinks,
linkcolor=black!75!red,
citecolor=blue,
pdftitle={},
pdfproducer={pdfLaTeX},
pdfpagemode=None,
bookmarksopen=true,
bookmarksnumbered=true,
backref=page]{hyperref}

\usepackage{tikz}
\usetikzlibrary{arrows,calc,decorations.pathreplacing,decorations.markings,intersections,shapes.geometric,through,fit,shapes.symbols,positioning,decorations.pathmorphing}

\usepackage{braket}

\usepackage[amsmath,thmmarks,hyperref]{ntheorem}
\usepackage{cleveref}

\creflabelformat{enumi}{#2(#1)#3}

\crefname{section}{Section}{Sections}
\crefformat{section}{#2Section~#1#3} 
\Crefformat{section}{#2Section~#1#3} 

\crefname{subsection}{\S}{\S\S}
\AtBeginDocument{%
  \crefformat{subsection}{#2\S#1#3}%
  \Crefformat{subsection}{#2\S#1#3}%
}

\crefname{subsubsection}{\S}{\S\S}
\AtBeginDocument{%
  \crefformat{subsubsection}{#2\S#1#3}%
  \Crefformat{subsubsection}{#2\S#1#3}%
}

\theoremstyle{plain}

\newtheorem{lemma}{Lemma}[section]
\newtheorem{proposition}[lemma]{Proposition}
\newtheorem{corollary}[lemma]{Corollary}
\newtheorem{theorem}[lemma]{Theorem}

\theoremstyle{plain}
\theoremnumbering{Alph}
\newtheorem{theoremN}{Theorem}

\theoremstyle{plain}
\theorembodyfont{\upshape}
\theoremsymbol{\ensuremath{\blacklozenge}}

\newtheorem{definition}[lemma]{Definition}
\newtheorem{example}[lemma]{Example}

\newtheorem{remark}[lemma]{Remark}

\crefname{definition}{definition}{definitions}
\crefformat{definition}{#2definition~#1#3} 
\Crefformat{definition}{#2Definition~#1#3} 

\crefname{ex}{example}{examples}
\crefformat{example}{#2example~#1#3} 
\Crefformat{example}{#2Example~#1#3} 

\crefname{exs}{example}{examples}
\crefformat{examples}{#2example~#1#3} 
\Crefformat{examples}{#2Example~#1#3} 

\crefname{remark}{remark}{remarks}
\crefformat{remark}{#2remark~#1#3} 
\Crefformat{remark}{#2Remark~#1#3} 

\crefname{remarks}{remark}{remarks}
\crefformat{remarks}{#2remark~#1#3} 
\Crefformat{remarks}{#2Remark~#1#3} 

\crefname{convention}{convention}{conventions}
\crefformat{convention}{#2convention~#1#3} 
\Crefformat{convention}{#2Convention~#1#3} 

\crefname{notation}{notation}{notations}
\crefformat{notation}{#2notation~#1#3} 
\Crefformat{notation}{#2Notation~#1#3} 

\crefname{table}{table}{tables}
\crefformat{table}{#2table~#1#3} 
\Crefformat{table}{#2Table~#1#3}

\crefname{lemma}{lemma}{lemmas}
\crefformat{lemma}{#2lemma~#1#3} 
\Crefformat{lemma}{#2Lemma~#1#3} 

\crefname{proposition}{proposition}{propositions}
\crefformat{proposition}{#2proposition~#1#3} 
\Crefformat{proposition}{#2Proposition~#1#3} 

\crefname{corollary}{corollary}{corollaries}
\crefformat{corollary}{#2corollary~#1#3} 
\Crefformat{corollary}{#2Corollary~#1#3} 

\crefname{theorem}{theorem}{theorems}
\crefformat{theorem}{#2theorem~#1#3} 
\Crefformat{theorem}{#2Theorem~#1#3} 

\crefname{enumi}{}{}
\crefformat{enumi}{(#2#1#3)}
\Crefformat{enumi}{(#2#1#3)}

\crefname{assumption}{assumption}{Assumptions}
\crefformat{assumption}{#2assumption~#1#3} 
\Crefformat{assumption}{#2Assumption~#1#3} 

\crefname{construction}{construction}{Constructions}
\crefformat{construction}{#2construction~#1#3} 
\Crefformat{construction}{#2Construction~#1#3} 

\crefname{equation}{}{}
\crefformat{equation}{(#2#1#3)} 
\Crefformat{equation}{(#2#1#3)}

\numberwithin{equation}{section}
\renewcommand{\theequation}{\thesection-\arabic{equation}}

\theoremstyle{nonumberplain}
\theoremsymbol{\ensuremath{\blacksquare}}

\newtheorem{proof}{Proof}
\newcommand\pf[1]{\newtheorem{#1}{Proof of \Cref{#1}}}

\newcommand\bA{{\mathbb A}}
\newcommand\bB{{\mathbb B}}
\newcommand\bC{{\mathbb C}}

\newcommand\bK{{\mathbb K}}

\newcommand\bQ{{\mathbb Q}}

\newcommand\bZ{{\mathbb Z}}

\newcommand\cM{{\mathcal M}}

\newcommand\cO{{\mathcal O}}

\newcommand\cS{{\mathcal S}}

\newcommand\fm{{\mathfrak m}}

\newcommand\fp{{\mathfrak p}}

\DeclareMathOperator{\End}{\mathrm{End}}
\DeclareMathOperator{\Hom}{\mathrm{Hom}}

\DeclareMathOperator{\lcm}{\mathrm{lcm}}

\newcommand\numberthis{\addtocounter{equation}{1}\tag{\theequation}}

\newcommand{\qedhere}{\mbox{}\hfill\ensuremath{\blacksquare}}

\title{High powers in endomorphism rings over Dedekind domains}
\author{Alexandru Chirvasitu}

\begin{document}

\date{}

\newcommand{\Addresses}{{
  \bigskip
  \footnotesize

  \textsc{Department of Mathematics, University at Buffalo}
  \par\nopagebreak
  \textsc{Buffalo, NY 14260-2900, USA}  
  \par\nopagebreak
  \textit{E-mail address}: \texttt{achirvas@buffalo.edu}


}}

\maketitle

\begin{abstract}
  Let $\mathbb{A}$ be a Dedekind domain and $T$ an endomorphism of a finitely-generated projective $\mathbb{A}$-module. If $T$ is an $s^{th}$ power in $\mathrm{End}_{\mathbb{A}}(M)$ for $s$ ranging over an infinite set $\mathcal{S}$ of positive integers, then (a) $T$ decomposes as a direct sum of the zero operator and an invertible operator on a summand of $M$ and (b) that summand is semisimple or of finite order if $\mathcal{S}$ is appropriately large (what this means depends on the structure of the additive and multiplicative groups of $\mathbb{A}$). This generalizes a result of M. Cavachi's to the effect that the only non-singular integer matrix that is an $s^{th}$ power in $M_n(\mathbb{Z})$ for all $s$ is the identity. 
\end{abstract}


\vspace{.5cm}



\section*{Introduction}

The original impetus for the note was provided by the remark \cite{cvch_11401} that the only non-singular integer-valued matrix that is an $n^{th}$ power of an integer matrix for every $n$ is the identity. Very short proofs exist (\cite[e.g. pp.934-935]{cvch_11401-sol}), but the problem suggests numerous follow-up questions:

\begin{enumerate}[(a)]
\item\label{item:a} Is it enough to assume the matrix is an $n^{th}$ power for just infinitely many $n$? {\it (no: $-1$ is a power with arbitrary odd exponent)};

\item How about an $n^{th}$ power for all but finitely many $n$? {\it (yes; most proofs generalize in this fashion)};

\item Assuming only infinitely many exponents, and taking a cue from \Cref{item:a} above, does it follow that the matrix is of finite order in the general linear group? {\it (yes; a consequence of \Cref{th:gla})};

\item If so, how does the order relate to the exponents in question? {\it (coprime to those primes dividing the exponents with arbitrarily high powers: \Cref{th:gla} reformulates this in terms of {\it supernatural numbers})};
  
\item What can one say if the matrix is singular? {\it (almost as much: it is diagonalizable over $\bZ$ to $\mathrm{diag}(0\cdots 0,\ 1\cdots 1)$; a consequence of \Cref{th:gla} again, but see also \cite{math.se_11401-gen} for idempotence)}. 
\end{enumerate}

More generally (and vaguely), it is tempting to abstract some of the arithmetic driving the phenomena above away from the specifics of the situation. To that end, the discussion below substitutes a Dedekind domain $\bA$ for the integers and an endomorphism $T$ of a finitely-generated projective $\bA$-module for the matrix. The main result (\Cref{th:gla}) disentangles several threads that appear entwined in the original problem:

\begin{theoremN}
  Let $\bA$ be a Dedekind domain and $T\in \End_{\bA}(M)$ an endomorphism of a finitely-generated projective $\bA$-module $M$.
  
  \begin{enumerate}[(1)]

  \item\label{item:splitN} If an endomorphism $T$ is an $s^{th}$ power in $\End_{\bA}(M)$ for arbitrarily large $s\in\bZ_{>0}$, then $T$ is the direct sum of the zero operator and an invertible operator on a summand of $M$. 
    
  \item\label{item:addN} Consider an infinite set $\cS$ of positive integers such that
    \begin{align*}
      \mathrm{char}(\bA)=0
      &\quad\xRightarrow{}\quad
        \text{the group $(A,+)$ has no non-trivial elements divisible by every $s\in \cS$;}\\
      \mathrm{char}(\bA)=p>0
      &\quad\xRightarrow{}\quad
        \left\{n\ |\ p^n\text{ divides some }s\in\cS\right\}
        \text{ is unbounded}.\numberthis\label{eq:unbddexp}
    \end{align*}
    If $T=T_s^s$, $T_s\in\End_{\bA}(M)$ for every $s\in \cS$, then the invertible summand of the preceding point is semisimple.

  \item\label{item:multN} Consequently, if the only elements of the multiplicative group $\bA^{\times}$ that are $s^{th}$ powers for all $s\in \cS$ are roots of unity, said invertible summand is in fact of finite order.

    Moreover, that order is coprime to every prime $p$ satisfying the right-hand condition in \Cref{eq:unbddexp}. 
  \end{enumerate}
\end{theoremN}

This makes it clear, in particular, that

\begin{itemize}
\item the direct-sum decomposition of part \Cref{item:splitN} is a rather general phenomenon, reminiscent of {\it Fitting}-type results (e.g. \cite[(19.16)]{lam_1st});

\item the semisimplicity of item \Cref{item:addN} stems from an ``additive'' constraint on the exponents;

\item while finally, the finite-order result in \Cref{item:multN} is a byproduct of a constraint on the multiplicative group $\bA^{\times}$ of units (which group is particularly simple when $\bA=\bZ$). 
\end{itemize}

All of this specializes well to rings of integers in algebraic number fields (\Cref{ex:glob0} and \Cref{cor:nrrings,cor:like-11401}), or in positive-characteristic global fields (\Cref{cor:globp}), as well as local fields of either positive (\Cref{ex:locp}) or vanishing (\Cref{ex:loc0}) characteristic. 

\subsection*{Acknowledgements}

I am grateful for valuable input from M. Cavachi, R. Kanda, M. Reyes and V. Trivedi.

This work is partially supported through NSF grant DMS-2011128.


\section{Highly divisible semisimple operators}\label{se:ss}

We assume some background on {\it Dedekind domains} (noetherian integrally closed domains of Krull dimension $\le 1$ \cite[\S I.3, Definition following Proposition 4]{ser_locf}), such as the reader can find in countless sources: \cite[Chapter 9]{am_comm}, \cite[\S 16.3]{df_3e}, \cite[\S I.3]{ser_locf}, \cite[Chapter 3]{marc_nfld}, etc. \cite[\S VII.2.2, Theorem 1]{bourb_commalg} and \cite[Theorem 6.20]{lm_mult} provide extensive lists of alternative characterizations.

\begin{remark}\label{re:fieldded}
  As defined here, the class of Dedekind domains includes that of fields; sources differ on this: \cite[Chapter 9]{am_comm}, \cite[\S I.3]{ser_locf}, \cite[Chapter 3]{marc_nfld} and \cite[\S I.3, Definition 1.3]{neuk_ant} agree, since they phrase the requirement via universal quantification over non-zero prime ideals, of which fields have none. On the other hand, because \cite[sentence following Theorem 9.3]{am_comm}, \cite[\S 16.3]{df_3e} and \cite[\S I.3, Definition following Theorem 6.2A]{hrt_ag} (for instance) require that the Krull dimension be {\it exactly} 1 (rather than only $\le 1$), the resulting Dedekind domains cannot be fields.

  Nothing below hinges crucially on the matter; having to make a choice for definiteness, we count fields among Dedekind domains.
\end{remark}

Recall in particular \cite[\S 16.3, Proposition 21 and Theorem 22]{df_3e} that an $\cO$-module $M$ is
  \begin{equation}\label{eq:dedek}
    \begin{aligned}
      &\text{finitely-generated {\it projective} \cite[\S 10.5, Definition preceding Corollary 31]{df_3e}} \iff\\
      &\text{it is finitely-generated torsion-free} \iff\\
      &M\cong \bigoplus_{s=1}^r I_s\text{ for ideals }I_s\trianglelefteq \cO \iff\\
      &M\cong \cO^{r-1}\oplus (I_1\cdots I_r),
    \end{aligned}
  \end{equation}
  where the last summand is the product of the $r$ ideals. If the $I_s$ of \Cref{eq:dedek} are non-zero, $r$ is the {\it rank} \cite[\S 12.1, Definition preceding Theorem 4]{df_3e} of $M$.
  
It will be convenient to use the language of {\it supernatural numbers} (\cite[\S 22.8]{fj_field}, \cite[\S 1.3]{ser_galcoh}, etc.): formal products $\prod_{p}p^{n_p}$ over primes $p$, with exponents $n_p\in\bZ_{\ge 0}\sqcup\{\infty\}$. For these, one can make sense in the obvious fashion of products, least common multiples and greatest common divisors, and other such arithmetic notions.

The usual {\it $p$-adic valuation} $\nu_p$ \cite[Example 2.2.1 (a)]{fj_field} attached to a prime number $p$ extends to supernatural numbers in the obvious fashion:
\begin{equation*}
  \nu_p\left(\prod_p p^{n_p}\right):= n_p.
\end{equation*}

We also borrow a piece of notation/terminology from \cite[\S 10.1]{ser_linrep}: for a set $\Pi$ of primes, a (supernatural) {\it $\Pi$-number} is one whose prime divisors all belong to $\Pi$, whereas a (supernatural) {\it $\Pi'$-number} is one whose prime divisors all lie outside of $\Pi$. 

Finally, we introduce some language in line with the standard terminology on {\it divisible groups (or modules)} \cite[\S 10.5, discussion preceding Proposition 36 and Example (4) following it]{df_3e}. 

\begin{definition}\label{def:div}
  Let $x\in \cM$ be an element in a monoid written multiplicatively.
  \begin{enumerate}[(1)]
  \item $x$ is {\it $s$-divisible (in $\cM$)} for a positive integer $s$ if there is $y\in \cM$ with $y^s=x$.

  \item Similarly, for a set $\cS$ of positive integers, $x$ is {\it $\cS$-divisible (in $\cM$)} if it is $s$-divisible for every $s\in \cS$.

  \item $x$ is {\it an arbitrarily high power} or {\it arbitrarily highly divisible} if it $\cS$-divisible for some infinite set $\cS$ of positive integers. 
  \end{enumerate}
\end{definition}

\begin{theorem}\label{th:gla}
  Let $\bA$ be a Dedekind domain and $M$ a finitely-generated projective $\bA$-module. 
  
  \begin{enumerate}[(1)]

  \item\label{item:split} If $T\in \End_{\bA}(M)$ is an arbitrarily high power, then $M=\ker T\oplus \mathrm{im}~T$ and $T|_{\mathrm{im}~T}$ is invertible. 
    
  \item\label{item:add} Consider an infinite set $\cS$ of positive integers such that
    \begin{align*}
      \mathrm{char}(\bA)=0
      &\quad\xRightarrow{}\quad
        \text{the group $(A,+)$ has no non-trivial $\cS$-divisible elements;}\numberthis\label{eq:add0}\\
      \mathrm{char}(\bA)=p>0
      &\quad\xRightarrow{}\quad
        p\in \Pi_{\cS}:=\left\{\text{primes p}\ |\ \nu_p\lcm \left(s\ |\ s\in \cS\right)=\infty\right\}.\numberthis\label{eq:addp}
    \end{align*}
    If $T\in\End_{\bA}(M)$ is $\cS$-divisible in $\End_{\bA}(M)$, then the restriction $T|_{\mathrm{im}(T)}$ of \Cref{item:split} is semisimple.

  \item\label{item:mult} If in addition $(A^{\times}/\mathrm{torsion}(A^{\times}),\cdot)$ also has no non-trivial $\cS$-divisible elements then for an $\cS$-divisible $T\in\End_{\bA}(M)$ the restriction $T|_{\mathrm{im}(T)}$ is of finite $\Pi_{\cS}'$-order.
    
  \item\label{item:from} Conversely, if $T$ is a direct sum of the zero operator and an operator of finite order $d$, then $T$ is an $n^{th}$ power in $\End_{\bA}(M)$ for every $n$ coprime to $d$.
  \end{enumerate}
\end{theorem}

The statement of \Cref{th:gla} \Cref{item:add} is phrased so as to have \Cref{eq:add0} plug directly into the proof, but that condition has an alternative, perhaps more transparent (because more directly numerical) description. 

\begin{definition}\label{def:lchars}
  The set of {\it local characteristics} of a domain $\bA$ is
  \begin{equation*}
    \mathrm{lchar}(\bA) := \left\{\mathrm{char}(\bA/\fp)\ |\ \{0\}\ne \fp\trianglelefteq \bA\text{ prime}\right\}.
  \end{equation*}
\end{definition}

\begin{proposition}\label{pr:addiv}
  For a Dedekind domain $\bA$ the conditions \Cref{eq:add0,eq:addp} are jointly equivalent to
  \begin{equation}\label{eq:sumnu}
    \sum_{p\in\mathrm{lchar}(\bA)}\sup_{s\in S}~\nu_{p}(s)=\infty.
  \end{equation}
\end{proposition}
\begin{proof}
  In positive characteristic $p$ the set $\mathrm{lchar}(\bA)$ is the singleton $\{p\}$, and \Cref{eq:sumnu} obviously rephrases \Cref{eq:addp}. Assuming henceforth that $\mathrm{char}(\bA)=0$, note that every prime $p\in \bZ_{>0}\subset \bA$ belongs to only finitely many prime ideals. For that reason, \Cref{eq:sumnu} can also be rendered as
  \begin{equation}\label{eq:sumnu'}
    \sum_{\text{primes }\fp\trianglelefteq \bA}\sup_{s\in S}~\nu_{\mathrm{char}(\bA/\fp)}(s)=\infty.
  \end{equation}
  Or, in words, (at least) one of the following two conditions obtains:
  \begin{enumerate}[(a)]
  \item\label{item:1p} there is some prime ideal $\fp\trianglelefteq \bA$ with
    \begin{equation*}
      \left\{\nu_p(s)\ |\ s\in\cS\right\}\text{ unbounded},\quad p:=\mathrm{char}(\bA/\fp);
    \end{equation*}
  \item\label{item:manyp} the set of prime ideals $\fp \trianglelefteq \bA$ containing some $s\in \cS$ is infinite. 
  \end{enumerate}  
  
  \begin{enumerate}[]
  \item {\bf \Cref{eq:add0} $\xRightarrow{\quad}$ \Cref{eq:sumnu'}:} The joint negation of \Cref{item:1p} and \Cref{item:manyp} means that there is a positive integer $n$ such that $\frac s{\gcd(s,n)}$, $s\in \cS$ belong to no prime ideals of $\bA$, and hence are invertible. $n\in\bZ\subseteq \bA$, then, will be $\cS$-divisible. 

  \item {\bf \Cref{eq:sumnu'} $\xRightarrow{\quad}$ \Cref{eq:add0}:} If \Cref{item:manyp} holds we are done, for an $\cS$-divisible element $x\in \bA$ would then belong to infinitely many prime ideals, as no non-zero $x$ can (since for $x\ne 0$ the principal ideal $(x)$ decomposes uniquely as a product finitely many prime ideals \cite[\S I.3, Corollary 3.9]{neuk_ant}).

    Assume \Cref{item:1p} holds instead. An $\cS$-divisible element is then $p^n$-divisible for every $n$, hence belongs to the trivial \cite[Corollary 10.18]{am_comm} intersection $\displaystyle \bigcap_n \fp^n\trianglelefteq \bA$.
  \end{enumerate}
\end{proof}

The setup of \Cref{th:gla} might appear somewhat contrived, but it covers (for appropriate $\cS$) the Dedekind domains of most interest in number theory: the rings of integers in either {\it local} or {\it global fields}. 

\begin{example}\label{ex:glob0}
  A {\it number field} $\bK$ is a finite extension of the rationals \cite[first sentence of Chapter 2]{marc_nfld}, which we may as well assume embedded in $\bC$. These are also the {\it global fields} of characteristic zero of \cite[\S II.12]{cf_ant_1967}. The corresponding {\it number ring} \cite[following Corollary 1 to Theorem 2]{marc_nfld} $\cO_{\bK}\subset \bK$, consisting of the algebraic integers in $\bK$, is a Dedekind domain \cite[Theorem 14]{marc_nfld}. 

  Any infinite $\cS$ will do: \Cref{eq:add0} obviously holds in its alternative incarnation as \Cref{eq:sumnu}, since $\mathrm{lchar}(\cO_{\bK})$ consists of {\it all} primes. As for the infinite-power property in the statement of \Cref{th:gla} \Cref{item:mult}, it follows from the fact that $\cO^{\times}$ is finitely generated as an abelian group (this is Dirichlet's celebrated {\it Unit Theorem}, usually stated much more precisely than we have any need to \cite[Theorem 38]{marc_nfld}). 
\end{example}

\begin{example}\label{ex:loc0}
  For a prime $p$, consider a finite extension $\bK$ of the field $\bQ_p$ of {\it $p$-adic numbers} \cite[\S II.1]{neuk_ant}. It is complete with respect to the unique extension $|\cdot|$ to $\bK$ \cite[\S II.4, Theorem 4.8]{neuk_ant} of the {\it $p$-adic norm } $|\cdot|_p$ of \cite[\S II.2]{neuk_ant}. Such $\bK$ are precisely the characteristic-0 {\it local fields} of \cite[\S II.5]{neuk_ant} (or \cite[Chapter VI, Introduction]{cf_ant_1967}). 

  The corresponding {\it discrete valuation ring}
  \begin{equation*}
    \cO_{\bK}:=\{x\in \bK\ |\ |x|\le 1\}
  \end{equation*}
  is a principal ideal domain \cite[\S I.1, Proposition 1]{ser_locf} (so in particular Dedekind). An infinite $\cS\subseteq \bZ_{>0}$ satisfies \Cref{eq:add0} if and only if $p\in \Pi_{\cS}$ (i.e. we can find $s\in S$ divisible by arbitrarily high powers of $p$), in which case the hypothesis of \Cref{th:gla} \Cref{item:mult} also holds.

  The first claim follows immediately from the fact that positive integers coprime to $p$ are invertible in $\cO_{\bK}$. To verify the second, recall the  direct-product decomposition (\cite[\S III.1, Proposition 1.1]{neuk_cft} or \cite[\S 15.1, (2')]{hasse_nt})
  \begin{equation}\label{eq:okuk}
    \cO_{\bK}^{\times}\cong (\text{finite cyclic group})\times U_{\bK}^{(1)},
  \end{equation}
  where the groups
    \begin{equation*}
      U_{\bK}^{(i)}:=1+\fm^i,\quad i\ge 1,\quad \fm\subset \cO_{\bK}\text{ is the unique maximal ideal }
    \end{equation*}
    are introduced in \cite[\S III.1]{neuk_cft} (also \cite[\S IV.2]{ser_locf} or \cite[\S 15.1]{hasse_nt}; in the latter, $\mathsf{H}_i=U_{\bK}^{(i)}$ and $\mathsf{H}=\mathsf{H}_1$). Similarly, 
    \begin{equation*}
      U_{\bK}^{(1)}\cong (\text{finite cyclic $p$-group})\times \bZ_p^{[\bK:\bQ_p]}
    \end{equation*}
    by \cite[\S XIV.4, Proposition 10]{ser_locf} or \cite[\S 15.5, One-unit theorem]{hasse_nt}, where $\bZ_p=\cO_{\bQ_p}$ is the ring of $p$-adic integers, regarded here as a group with its additive structure. All in all,
  \begin{equation*}
    \cO_{\bK}^{\times}\cong F\times \bZ_p^{[\bK:\bQ_p]},\quad F\text{ finite abelian},
  \end{equation*}
  whence the conclusion.
\end{example}

\begin{example}\label{ex:locp}
  The substance of the discussion in \Cref{ex:loc0} goes through ( that is, \Cref{th:gla} \Cref{item:mult} applies precisely when $p\in \Pi_{\cS}$) for rings of integers in {\it positive}-characteristic local fields: per \cite[Chapter VI, Introduction]{cf_ant_1967}, the fields $\bK=\Bbbk((t))$ of Laurent power series over finite fields $\Bbbk$ (whereupon $\cO_{\bK}=\Bbbk[[t]]$, the ring of formal power series).

  \Cref{eq:okuk} holds just as before, since the cited references are characteristic-blind on that count. As for $U_{\bK}^{(1)}$, it is this time simply a free $\bZ_p$-module \cite[\S 15.4, One-unit theorem]{hasse_nt} (albeit one of infinite rank this time). 
\end{example}

An application of \Cref{th:gla} to \Cref{ex:glob0} yields 

\begin{corollary}\label{cor:nrrings}
  Let $M$ be a finitely-generated projective module over a number ring $\cO_{\bK}$ and $\cS$ an infinite set of positive integers.

  An $\cS$-divisible one-to-one $T\in\End_{\cO_{\bK}}(M)$ is of finite order coprime to every $p\in \Pi_{\cS}$. In particular, $T=1$ provided for every prime $p$, there are elements of $\cS$ divisible by arbitrarily high powers of $p$. 
\end{corollary}
\begin{proof}
  \Cref{ex:glob0} notes that parts \Cref{item:add} and \Cref{item:mult} of \Cref{th:gla} apply to any infinite $\cS$, and the non-singularity condition disposes of $\ker T$.
\end{proof}

Specializing \Cref{cor:nrrings} further to $M:=\cO_{\bK}^m$ provides the following generalization of \cite{cvch_11401} (which in turn can be recovered by setting $\bK=\bQ$):

\begin{corollary}\label{cor:like-11401}
  Let $\cO_{\bK}$ be a number ring and $m$ a positive integer. The only non-singular matrix in $M_m(\cO_{\bK})$ that is an $n^{th}$ power therein for all but finitely many $n$ is the identity.  \qedhere
\end{corollary}

We will also consider Dedekind domains $\bA$ whose quotient fields are finite extensions of $\Bbbk(t)$ (the fields featuring in the definition of an {\it abstract smooth curve} \cite[\S I.6, following Corollary 6.6]{hrt_ag}), for positive-characteristic $\Bbbk$. When $\Bbbk$ is finite these are also the positive-characteristic global fields \cite[\S II.12]{cf_ant_1967}, ``globalizing'' \Cref{ex:locp} akin to the passage from \Cref{ex:loc0} to \Cref{ex:glob0}. 

\begin{corollary}\label{cor:globp}
  Let $\bA$ be a Dedekind domain whose field of fractions $\bK$ is a finite extension of $\Bbbk(t)$ for $p:=\mathrm{char}(\Bbbk)>0$, and $M$ a finitely-generated projective $\bA$-module.

  \begin{enumerate}[(1)]
  \item If $T\in \End_{\bA}(M)$ is $\cS$-divisible for an infinite $\cS\subseteq \bZ_{>0}$ with $p\in \Pi_{\cS}$ then $T$ is diagonalizable over the algebraic closure $\overline{\Bbbk}$ of $\Bbbk$. 
    
  \item In particular, $T=0\oplus T'$ with $T'$ of $p$-coprime finite order if the only $\cS$-divisible roots of unity in $\Bbbk^{\times}$ are roots of unity (e.g. if $\Bbbk$ is finite or, more generally, algebraic over its prime field). 
  \end{enumerate}
\end{corollary}
\begin{proof}
  All of this follows from \Cref{th:gla} (and its proof) upon noting that the arbitrarily highly divisible elements of $\bK^{\times}$ must be algebraic over $\Bbbk\subset \Bbbk(t)\subseteqq \bK$.
\end{proof}

The direct-sum decomposition of \Cref{th:gla} \Cref{item:split} is fairly easily dispatched. It relies in part on the following simple general remark, itself a variant of the {\it Fitting lemma} (variants of which appear as \cite[\S 15.1, Exercise 5]{df_3e}, \cite[\S 3.3, preceding Theorem 3.7]{jac_basic-2_2e}, etc.):

\begin{lemma}\label{le:fitt}
  Let $M$ be a noetherian module over a commutative ring $\bA$ and $T\in\End_{\bA}(M)$.

  If the endomorphism $\overline{T}$ induced by $T$ on $M/\ker T^n$ is onto for some $n$, then $M=\ker T^m\oplus \mathrm{im}~T^m$ and $T|_{\mathrm{im}~T^m}$ is an automorphism for $m\gg 0$. 
\end{lemma}
\begin{proof}
  The already-cited \cite[\S 15.1, Exercise 5]{df_3e} shows that
  \begin{itemize}
  \item the non-decreasing chain of submodules $\ker T^m$ stabilizes;
  \item the sum
    \begin{equation}\label{eq:kerim}
      \ker T^m+ \mathrm{im}~T^m \le M,\quad m\gg 0
    \end{equation}
    is direct;
  \item and $\overline{T}$ is in fact an {\it auto}morphism of $M/\ker T^m$, $m\gg 0$. 
  \end{itemize}
  The conclusion follows immediately:
  \begin{equation*}    
    \ker T^m+\mathrm{im}~T^m/\ker T^m
    =
    \mathrm{im}~\overline{T}^m
    =
    \mathrm{im}~\overline{T}
    =
    \ker T^m+\mathrm{im}~T/\ker T^m
    =
    M/\ker T^m,
  \end{equation*}
  so \Cref{eq:kerim} cannot be proper. 
\end{proof}

\begin{lemma}\label{le:unip1}
  Let $\bA$ be a Dedekind domain with quotient field $\bK$, $M$, $T$ and $\cS$ as in \Cref{th:gla}, and assume \Cref{eq:add0,eq:addp}. Denote also by $\bA\subseteq \overline{\bA}\subset \overline{\bK}$ the integral closure of $\bA$ in the algebraic closure $\overline{\bK}\supseteq \bK$.

  If $T\in\End_{\bA}(M)$ is unipotent and $\cS$-divisible in $\End_{\overline{\bA}}(M\otimes_{\bA} \overline{\bA})$ then it is the identity.  
\end{lemma}
\begin{proof}
  Set $E:=\End_{\bA}(M)$ and denote by subscripts modules obtained by scalar extension: $M_{\bK}:=M\otimes_{\bA}\bK$,
  \begin{equation*}
    E_{\overline{\bA}}:=E\otimes_{\bA}\overline{\bA}\cong \End_{\overline{\bA}}\left(M_{\overline{\bA}}\right),
  \end{equation*}
  and so on. 

  Fix $T_s\in E_{\overline{\bA}}$ with $T_s^s=T$, $s\in \cS$. The eigenvalues of $T_s$ are roots of unity (since those of $T$ are 1: this is what {\it unipotence} \cite[\S I.4]{brl_lalgps} means). It follows that the semisimple factor $R_s$ in the {\it multiplicative Jordan decomposition} \cite[\S I.4, Corollary 1 to Proposition 4.2]{brl_lalgps} $T_s=R_sU_s$ belongs to $E_{\overline{\bA}}\subset E_{\bK}= \End_{\bK}(M_{\bK})$ along with its inverse, so that $U_s\in E_{\overline{\bA}}$ as well. Working with $U_s$ in place of $T_s$, we may now assume the latter unipotent. 

  We argue inductively on the minimal $n$ with $(T-1)^n=0$, with the inductive step consisting of substituting $M':=M/\ker(T-1)$ (also torsion-free) for $M$ and replacing $T$ and $T_s$ with the operators induced thereon. It will thus be enough to assume that $(T-1)^2=0$ (the base case of the induction).
  
  $M'$ (because it is finitely-generated torsion-free) being projective, there is a (non-canonical) decomposition $M\cong \ker (T-1)\oplus M'$ that transports over to $M_{\overline{\bA}}$ and gives block upper-triangular decompositions
  \begin{equation*}
    T=
    \begin{pmatrix}
      1&S\\
      0&1
    \end{pmatrix}
    ,\quad
    T_s=
    \begin{pmatrix}
      1&S_s\\
      0&U_s
    \end{pmatrix}
    ,\quad
    U_s^s=1.
  \end{equation*}
  Consider the two cases:
      \begin{enumerate}[(a)]

      \item In characteristic 0 the $U_s$ are identities (being both unipotent and of finite order) so that
        \begin{equation*}
          \begin{aligned}
            T_s^s=T
            &\ \xRightarrow{\quad}\ 
              sS_s = S\\
            &\ \xRightarrow{\quad}\ 
              S\text{ is $\cS$-divisible in } \Hom_{\overline{\bA}}\left(M'_{\overline{\bA}},\ \ker(T-1)_{\overline{\bA}}\right)\cong \Hom_A(M',\ker(T-1))_{\overline{\bA}}.
          \end{aligned}          
        \end{equation*}
        Since the morphism space is projective finitely-generated over $\bA$, the latter's assumed $\cS$-non-divisibility implies that $S$ vanishes and hence $T=1$.

      \item In characteristic $p>0$ we still have
        \begin{equation*}
          U_s^{p^{\nu_p(s)}}=1,\ \forall s\in \cS,        
        \end{equation*}
        since those powers of $U_s$ are both unipotent and roots of unity of orders $\frac{s}{p^{\nu_p(s)}}$ (coprime to $p$). Because the $U_s$ all operate on the same finite-dimensional vector space $M_{\bK}$, there is some $m$ such that
        \begin{equation*}
          U_s^{p^m}=1,\quad \forall s\in \cS. 
        \end{equation*}
        Our assumption that $p\in \Pi_{\cS}$ implies that $\nu_p(s)>m$ for at least one $s\in \cS$; $S$ then vanishes, being a multiple of the $p$-divisible $\frac s{p^m}$.
      \end{enumerate}
      This concludes the proof. 
\end{proof}

\pf{th:gla}
\begin{th:gla}
  \begin{description}
  \item {\bf \Cref{item:from}} is immediate: if $M\cong \ker T\oplus P$ with $T|_P$ of order $d$, then $T=(T^m)^n$ whenever $mn=1\mod d$; if $n$ and $d$ are coprime then such an $m$ always exists, hence the conclusion.

  \item {\bf \Cref{item:split}:} Note first that if $T\in \End_{\bA}(M)$ is nilpotent and arbitrarily highly divisible, then it vanishes. Indeed, the operators $T$ and
    \begin{equation*}
      T_s\in\End_{\bA}(M),\quad T_s^s=T,\quad \forall s\in \cS
    \end{equation*}
    on the $r$-dimensional ($r:=\mathrm{rank}(M)$) vector space $M_{\bK}:=M\otimes_{A}\bK$ over the quotient field $\bK$ of $\bA$ are all nilpotent, so \cite[\S 12.3, Exercise 32]{df_3e} $T_s^{r}=0$, $\forall s$. But then $T=T_s^{s}$ vanishes as soon as $s\ge r$.  
    
    In general, the preceding argument shows that the restriction of $T$ to the {\it generalized kernel}
    \begin{equation*}
      \ker_{gen} T:=\{v\in M\ |\ T^nv=0\text{ for some }n\}
    \end{equation*}
    vanishes, so that $\ker_{gen}T=\ker T$. But then $\ker T$ and $\mathrm{im}~T$ already intersect trivially and the sum
    \begin{equation*}
      \ker T+\mathrm{im}~T\le M
    \end{equation*}
    is direct. We will then be able to conclude via \Cref{le:fitt} (and its proof) as soon as we argue that the operator $\overline{T}$ induced by $T$ on $\overline{M}:=M/\ker T$ is onto (and hence invertible).

    To see this, note that $\overline{M}$ is again projective finitely generated, so that one can speak of determinants of operators thereon. Now, the principal ideal $(\det T)\trianglelefteq \bA$ is an arbitrarily high power in the multiplicative group of {\it fractional ideals} \cite[\S I.3, Definition 3.7]{neuk_ant}:
    \begin{equation*}      
      (\det T) = (\det T_s)^{s}
      \ \trianglelefteq\ 
      \bA
      ,\quad
      s\in \cS.
    \end{equation*}
    That group being free abelian on the set of prime ideals \cite[\S I.3, Corollary 3.9]{neuk_ant}, it follows that $\det T$ is invertible in $\bA$.   
    
    Restricting $T$ and all of the $T_s$ to the summand $\mathrm{im}~T\le M$, we can (and throughout the remainder of the proof will)  assume $T$ invertible. 

  \item {\bf \Cref{item:add}:} Assuming invertibility, we prove semisimplicity. Extend $\bK$ to an overfield $L$ by adjoining the eigenvalues $\alpha_i$ of $T$. Those eigenvalues are integral over $\bA$, by the familiar argument (via \cite[Proposition 5.1]{am_comm}, say): $\End_{\bA}(M)$ is finitely-generated as an $\bA$-module, hence so is the $\bA$-submodule generated as an $\bA$-algebra by $T$. In other words $\alpha_i\in \bB$, the integral closure of $\bA$ in $L$, itself a Dedekind domain \cite[\S I.12, Proposition 12.8]{neuk_ant}.

    Observe next that the hypothesis of \Cref{th:gla} \Cref{item:add} transports over from $\bA$ to $\bB$: for \Cref{eq:addp} this is clear, since the two rings have the same characteristic, while for \Cref{eq:add0} the claim follows from \Cref{pr:addiv} and the fact that every prime $\fp\trianglelefteq \bA$ is contained in (and the intersection of $\bA$ with) finitely many $\mathfrak{P}_i\trianglelefteq \bB$ (\cite[\S I.8, following Proposition 8.1]{neuk_ant} or \cite[Chapter 5, Exercise 15]{am_comm}). 
    
    The upshot of all of this is that we may substitute $\bB$ and $L$ for $\bA$ and $\bK$ respectively, or, what is more alphabetically economical, simply assume that $\alpha_i\in \bA$. But then the factors of the multiplicative Jordan decomposition $T=R U$ both belong to $\End_{\bA}(M)$, those of the analogous factorizations $T_s=R_s U_s$ belong to
    \begin{equation*}
      \End_{\overline{\bA}}(M\otimes_{\bA}\overline{\bA})
      ,\quad
      \overline{\bA}:=\text{integral closure of $\bA$ in the algebraic closure }\overline{\bK}\supseteq \bK,
    \end{equation*}
    and we can conclude by applying \Cref{le:unip1} to the $\cS$-divisibility $U=U_s^s$, $s\in \cS$ of the unipotent operator $U$ (in place of $T$) that $U=1$.

  \item {\bf \Cref{item:mult}:} Because we are assuming that the only $\cS$-divisible elements of $\bA^{\times}$ are roots of unity, so is $\det T$ and hence also the eigenvalues of $T$. But then $T$ is also semisimple by part \Cref{item:add}, hence the finite-order claim. 

    As to the constraint on the order of $T$: for every $p\in \Pi_{\cS}$ there is some $s$ for which lifting to the $s^{th}$ power annihilates the entire {\it $p$-primary component} of $\mathrm{torsion}(A^{\times})$ (i.e. the group of elements whose order is a power of $p$ \cite[\S 4.5, Example (2) following Corollary 20]{df_3e}), and hence the order of every $s^{th}$ power is coprime to $p$. The conclusion follows from
      \begin{equation*}
        (\det T_s)^{s} = \det T,\ \forall s.
      \end{equation*}    
  \end{description}
  This finishes the proof. 
\end{th:gla}

It is perhaps worth noting that occasionally, the multiplicative constraint of \Cref{th:gla} \Cref{item:mult} {\it follows} from the hypothesis of part \Cref{item:add}:

\begin{proposition}\label{pr:add2mult}
  Let $\cS\subseteq \bZ_{>0}$ and $\bA$ a Dedekind domain with at least one finite residue field of characteristic $p\in \Pi_{\cS}$.

  The hypotheses of \Cref{th:gla} \Cref{item:add} and \Cref{item:mult} are then met.
\end{proposition}
\begin{proof}
  The positive-characteristic branch \Cref{eq:addp} is obvious, \Cref{eq:add0} holds also by \Cref{pr:addiv}, and the hypothesis of \Cref{th:gla} \Cref{item:mult} (the fact that the $\cS$-divisible elements of $\bA^{\times}$ are roots of unity) follows from the corresponding claim for integer rings of local fields (\Cref{ex:loc0,ex:locp}) and the embedding \cite[Remark (1) following Theorem 10.17]{am_comm}
  \begin{equation*}
    \bA
    \lhook\joinrel\xrightarrow{\quad}
    \text{localization }\bA_{\fp}
    \lhook\joinrel\xrightarrow{\quad}
    \text{$\fp$-adic completion }\widehat{\bA}_{\fp}:=\varprojlim_n \bA/\fp^n 
  \end{equation*}
  for some prime ideal $\fp\trianglelefteq \bA$ with finite characteristic-$p$ residue field $\bA/\fp$. 
\end{proof}


\addcontentsline{toc}{section}{References}

\Addresses

\end{document}